\documentclass[12pt]{article}
\usepackage{amsthm}
\usepackage{amssymb}
\usepackage{amsmath}
\usepackage{latexsym}
\numberwithin{equation}{section}
\def\o{\prec}
\def\oo{\succ}
\def\S{\underline{S}}
\def\s{\underline{s}}
\def\s{\underline{s}}
\def\loc{\mathrm{loc}}
\def\R{\mathbb R}
\def\C{\mathbb C}
\def\N{\mathbb N}
\def\Z{\mathbb Z}
\def\re{\operatorname{Re}}
\def\im{\operatorname{Im}}
\def\area{\operatorname{area}}
\def\diam{\operatorname{diam}}
\def\dim{\operatorname{dim}}

\def\essinf{\operatorname*{ess\,inf}}
\def\length{\operatorname{length}}
\newtheorem{theorem}{Theorem}
\newtheorem{proposition}{Proposition}
\newtheorem{lemma}{Lemma}
\newtheorem*{thma}{Theorem A}
\newtheorem*{thmb}{Theorem B}
\newtheorem*{thmc}{Theorem C}
\newtheorem*{thmd}{Theorem D}
\newtheorem*{thme}{Theorem E}
\newtheorem*{thmf}{Theorem F}
\newtheorem*{thmg}{Theorem G}
\begin{document}
\title{Dynamics of a higher dimensional analog of the trigonometric functions}
\author{Walter Bergweiler\thanks{Supported by the EU Research
Training Network CODY, the
ESF Networking Programme HCAA
and the Deutsche Forschungsgemeinschaft, Be 1508/7-1.} 
\ and Alexandre Eremenko\thanks{Supported by NSF grant
DMS--0555279 and by the Humboldt Foundation.}
}
%%%\subjclass{30D05, 30C65, 37F10, 37F35}
%%%\keywords{Dynamics of entire functions, quasiregular maps}
%%%\date{\today}
\date{}
\maketitle
\begin{abstract}
We introduce a quasi\-regular analog
$F$ of the sine and cosine function such that,
for a sufficiently large
constant $\lambda$, the map  $x \mapsto \lambda F(x)$ is
locally expanding.
We show that the dynamics of this map
define a representation
of $\R^d$, $d\geq 2$,
as a union of simple curves $\gamma:[0,\infty)\to
\R^d$
which tend to $\infty$ and
whose interiors $\gamma^*=\gamma((0,\infty))$ are disjoint 
such that the union of all
$\gamma^*$ has Hausdorff dimension~$1$.
\end{abstract}

\section{Introduction and statement of results}\label{intro}
The Julia set $J(f)$ of an entire function $f$
is defined as the set of all points in $\C$
where the iterates $f^{k}$ of $f$ do not form a normal
family. An equivalent definition was given in~\cite{E}:
$J(f)=\partial I(f)$ where 
$I(f)=\{ z:f^n(z)\to\infty\}$
is the set of escaping points; see~\cite{Ber93} for
an introduction to the dynamics of entire an meromorphic
functions.

Devaney and Krych \cite{DK} showed that $J(\lambda e^z)$
is a ``Cantor bouquet'' for $0<\lambda<1/e$.
To give a precise statement of their result we say
that a subset $H$ of $\C$ (or $\R^d$) is a {\em hair}
if there exists a continuous injective map
$\gamma:[0,\infty)\to\C$ (or $\R^d$) such that
$\lim_{t\to\infty}\gamma(t)=\infty$
and $\gamma([0,\infty))=H.$ We call $\gamma(0)$ the
{\em endpoint} of the hair.

The result of Devaney and Krych is the following.
\begin{thma}
Let $0<\lambda< 1/e$. Then
$J(\lambda e^z)$ is an uncountable union
of pairwise disjoint hairs.
\end{thma}

We denote by $\dim X$ the Hausdorff dimension of a set $X$ in
$\C$ (or in~$\R^d$). The following result is due to
McMullen~\cite[Theorem~1.2]{Mc}.
\begin{thmb}
For $\lambda\in\C\backslash\{0\}$ we
have $\dim J(\lambda e^z)=2.$
\end{thmb}
Karpi\'nska~\cite[Theorem~1.1]{K} proved the following surprising result.
\begin{thmc}
Let $0<\lambda<1/e$ and let $E_\lambda$
be the set of endpoints of the hairs
that form $J(\lambda e^z)$.
Then $\dim E_\lambda=2$ and $\dim(J(\lambda e^z)\backslash
E_\lambda)=1$.
\end{thmc}
The conclusion of Theorem B holds more generally
for entire functions of finite order
for which the set of critical and asympotic values is 
bounded; see~\cite[Theorem~A]{Bar08} and~\cite{Schubert07}.
If, in addition, this set is
compactly contained in the immediate basin of an attracting fixed 
point, then the conclusions of Theorems A and C also hold~\cite{Bar07,Bar08}.

These results apply in particular to trigonometric functions.
However, the analogue of Theorem A for trigonometric functions
had been obtained already much earlier by
Devaney and Tangerman~\cite{DT}.
\begin{thmd}
Let $0<\lambda< 1$. Then 
$J(\lambda\sin z)$ is an uncountable union of pairwise
disjoint hairs.
\end{thmd}
McMullen~\cite[Theorem~1.1]{Mc} and Karpi\'nska~\cite[Theorem~3]{Karp1} 
also considered the case of
trigonometric functions. Their results are as follows.
Here $\area X$ stands for the Lebesgue measure of a measurable subset 
$X$ of~$\C$.
\begin{thme}
Let $\lambda,\mu\in\C,\;\lambda\neq 0$. Then $\area J(\lambda\sin z+\mu)>0$.
\end{thme}
\begin{thmf}
For $0<\lambda<1$ let $E_\lambda$
be the set of endpoints of hairs that form
$J(\lambda\sin z)$.
Then 
$\area E_\lambda>0$.
\end{thmf}
The argument in~\cite{K} shows that under the hypothesis of 
Theorem~F we also have 
$\dim(J(\lambda \sin z)\backslash E_\lambda)=1$.

The conclusions of Theorems $D$ and $F$, as well as the last remark,
hold more generally
for functions of the form $f(z)=\lambda\sin z+\mu$ if the
parameters are chosen such that the critical values
$\pm\lambda+\mu$ of $f$ are contained in the immediate
basin of the same attracting fixed point.
If this condition on the critical values 
is not satisfied,
then the hairs in the Julia set of $f$ still may exist,
but in general distinct hairs may share their endpoints~\cite{RS}.

If the critical values of $f(z)=\lambda\sin z+\mu$
are strictly preperiodic, then $J(f)=\C$.
Schleicher~(\cite{S}, see also~\cite{S1}) showed that $J(f)$ is still
a union of hairs which are pairwise disjoint except
for their endpoints, and the Hausdorff dimension of
the hairs without their endpoints is~$1$.
Thus he obtained the following result.
\begin{thmg}
There exists a representation of 
the complex plane $\C$
as a union of
hairs with the following properties:
\begin{itemize}
\item the intersection of two hairs is either empty
or consists of the common endpoint;
\item the union of the hairs without their endpoints has
Hausdorff dimension~$1$.
\end{itemize}
\end{thmg}
Zorich~\cite{Z} introduced a quasi\-regular analog
$F:\R^3\to\R^3\backslash\{0\}$ of the exponential
function. It was shown in~\cite{B} that the results
about the dynamics of the exponential function quoted above
(Theorems A, B and~C) have analogs in the
context of Zorich maps.

In this paper we introduce a higher dimensional analog of
the trigonometric functions. The dynamics of this map
are then used to extend Theorem G to all dimensions
greater than~$1$.
\begin{theorem}
For each $d\in\N$, $d\geq 2$,
there exists a representation of $\R^d$ as a union
of hairs with the following properties:
\begin{itemize}
\item the intersection of two hairs is either empty
or consists of the common endpoint;
\item the union of the hairs without their endpoints has Hausdorff
dimension~$1$.
\end{itemize}
\end{theorem}
The construction of our higher dimensional analog of the
trigonometric functions 
is similar to the construction of Zorich's map
as given in~\cite[Section 6.5.4]{Iwaniec01}.
We begin with a bi-Lipschitz map $F$ from the half-cube
$$
\left\{x=(x_1,\dots,x_d)\in\R^d:\|x\|_\infty\leq 1, x_d\geq 0 \right\}
=[-1,1]^{d-1}\times [0,1]$$
to the upper half-ball
$$\{x\in\R^d:\|x\|_2 \leq 1,x_d\geq 0\}$$
which maps the 
face $[-1,1]^{d-1}\times \{1\}$ 
to the hemisphere
$\{x\in\R^d:\|x\|_2 = 1,x_d\geq 0\}$.
We will give an explicit construction of such a bi-Lipschitz
map $F$ in
Section~\ref{explicit}.
Next we define $F:[-1,1]^{d-1}\times (1,\infty)\to\R^d$ by
$$F(x)=\exp(x_d-1)F(x_1,\dots,x_{d-1},1).$$
The map $F$ is now defined on 
$[-1,1]^{d-1}\times [0,\infty)$, and it maps 
$[-1,1]^{d-1}\times [0,\infty)$
bijectively
onto the upper half-space $H^+:=\{x\in\R^d:x_d\geq 0\}$.
Using repeated reflections at hyperplanes we 
can extend $F$ to a map $F:\R^d\to\R^d$.

It turns out 
that the map $F$ is quasi\-regular.
However, we shall not actually use this fact.
On the other hand, the quasi\-regularity of $F$ is one of the 
underlying ideas in the proofs, and thus we make some remarks 
about quasiregular maps
in Section~\ref{quasi}.
We also show there that our map $F$ is indeed quasiregular.

We note that since $F$ is locally bi-Lipschitz, 
the restriction of $F$ to any line is absolutely continuous, and
$F$ is differentiable almost
everywhere. We denote by
$$\|DF(x)\|:=\sup_{\| y\|=1}\| DF(x)(y)\|$$
the operator norm of the derivative $DF(x)$.
(Here and in the following  $\|y\|=\|y\|_2$ for $y\in\R^d$; 
that is, unless specified otherwise we consider the
Euclidean norm in $\R^d$.)
We also put
$$\ell(DF(x)):=\inf_{\| y\|=1}\| DF(x)(y)\|.$$
We note that it follows from the definition of $F$ that
if $x,x' \in (-1,1)^{d-1}\times (1,\infty)$ and $x_j=x_j'$ for
$1\leq j\leq d-1$, then 
\begin{equation}\label{DFexp}
DF(x')=\exp(x_d'-x_d) DF(x)
\end{equation}
 whenever these derivatives exist.

It is easy to see that 
$$\beta:=\essinf_{ x\in\R^d}\ell(DF(x))>0$$
for our map~$F$. We choose $\lambda >1/\beta$ and consider
the map $f=\lambda F$. Clearly $f$ is quasi\-regular and
\begin{equation}
\label{1a}
\alpha:=\essinf_{x\in\R^d}\ell(Df(x))=\lambda \beta>1,
\end{equation}
that is, $f$ is locally uniformly expanding in $\R^d$.

We put $S:=\Z^{d-1}\times\{-1,1\}$ and for 
$r=(r_1,\dots,r_d)\in S$ we define
$$T(r):=\{x\in\R^d:|x_j-2r_j|\leq 1 \text{ for } 1\leq j\leq d-1,\;
r_dx_d\geq 0\}.$$
We find that if 
$$\sigma(r):=\sum_{j=1}^{d-1}r_j+\frac12(r_d-1)$$ 
is even, then $f$ maps
$T(r)$ bijectively onto $H^+$.
If $\sigma(r)$ is odd, then
$f$ maps $T(r)$ bijectively onto $H^-:=\{x\in\R^d:x_d\leq 0\}$.

For a sequence $\s=(s_k)_{k\geq 0}$
of elements of $S$ we put
$$H(\s):=\left\{ x\in\R^d:f^k(x)\in T(s_k)\ \mbox{for all}\ k\geq 0
\right\}.$$ 
Evidently $\R^d=\sum_{\s\in \S} H(\s)$,
where $\S$ is the set of all sequences with elements in~$S$
for which $H(\s)$ is not empty.
\begin{proposition}\label{prop1}
If $\s\in\S$ then
$H(\s)$ is a hair.
\end{proposition}
For $\s\in\S$ we denote by $E(\s)$ the endpoint of $H(\s)$.
\begin{proposition}\label{prop2}
If $\s'\neq\s''$ then 
$H(\s')\cap H(\s'')=\emptyset$ or $H(\s')\cap H(\s'')=\{E(\s)\}$.
\end{proposition}
\begin{proposition}\label{prop3}
$\dim\left( \bigcup_{\s\in \S} H(\s)\backslash \{ E(\s)\}\right)=1$.
\end{proposition}
Theorem 1 follows from these propositions.
\section{Preliminaries}
It follows from the definition of $F$ that
$$\| F(x)\|= \exp(|x_d|-1),\quad x
\in\R^d,\ |x_d|\geq 1,$$
so that
\begin{equation}\label{2b}
\| f(x)\|= \lambda \exp(|x_d|-1),\quad x
\in\R^d,\ |x_d|\geq 1.
\end{equation}
For $r\in S$ we denote by $\Lambda^r$ the inverse function
of $f|_{T(r)}$. Thus $\Lambda^r:H^+\to T(r)$
or $\Lambda^r:H^-\to T(r)$, depending on
whether $\sigma(r)$ is even or odd. 
For $x\in T(r)$ and $y=f(x)$ we have
$$\| D\Lambda^r(y)\|=\frac{1}{\ell(Df(x))}$$
and thus
\begin{equation}\label{n1}
\| D\Lambda^r(y)\|\leq \frac{1}{\alpha}
\end{equation}
by~\eqref{1a}.
It follows from~\eqref{n1} that
if $a,b\in T(r)$, then
$$\| a-b\|=\left\|\Lambda^r(f(a))-\Lambda^r(f(b))\right\|
\leq \frac{1}{\alpha}\| f(a)-f(b)\|.$$
Hence
\begin{equation}\label{2a}
\| f(a)-f(b)\|\geq\alpha\| a-b\|\quad\mbox{for}\ 
a,b\in T(r),\  r\in S.
\end{equation}
If $|x_d|\geq 1$ then we have 
$$\ell(Df(x))\geq \alpha \exp(|x_d|-1)=\frac{\alpha\| f(x)\|}{\lambda}
=\beta \| f(x)\|$$
by~\eqref{DFexp}, \eqref{1a} and~\eqref{2b}. 
Note that the condition  $|x_d|\geq 1$ is equivalent to
$\|y\|\geq  \lambda$.
Thus
\begin{equation}\label{n2}
\| D\Lambda^r(y)\|\leq\frac{1}{\beta \| y\|},\quad 
y\in\R^d,\ \|y\|\geq  \lambda.
\end{equation}
Similarly we deduce from~\eqref{DFexp} that there exists a
positive constant $\delta$ such that
\begin{equation}\label{n3}
\ell\left( D\Lambda^r(y)(x)\right)\geq \frac{\delta}{\| y\|} ,\quad
y\in\R^d,\ \|y\|\geq  \lambda.
\end{equation}

We shall also need the following result.
\begin{lemma}\label{lemma1}
Let $\s=(s_k)_{k\geq 0}$
be an element of $\S$ and let $x,y\in H(\s)$. For $k\geq 0$
we put $x^k=(x_1^k,\dots,x_d^k):=f^k(x)$ and
$y^k=(y_1^k,\dots,y_d^k):=f^k(y)$.

There exists $M>0$ with the following property: if
\begin{equation}
\label{2c}
|y_d^k|>|x_d^k|+M
\end{equation}
for some $k\geq 0$ then
$$ |y_d^{k+1}|>\frac{\lambda }{3}\exp{|y^k_d|}+M\geq 5|x_d^{k+1}|+M.$$
\end{lemma}
\noindent
{\em Proof.}  We will denote by $p$ the projection
$$p:\R^d\to\R^{d-1},\ 
 (x_1,\dots,x_{d-1},x_d)\mapsto(x_1,\dots,x_{d-1}).$$
Since $|x^k_j-y^k_j|\leq 2$ for $1\leq j\leq d-1$
and all $k$ we have 
\begin{equation}
\label{2c1}
\left\| p(x^{k})- p(y^{k})\right\|
\leq 2\sqrt{d-1}
\end{equation}
for all~$k$.

Suppose now that (\ref{2c}) holds. 
Then using (\ref{2b}) and (\ref{2c1})
we obtain
\begin{eqnarray*}
\left|y^{k+1}_d\right|&\geq&\left\| y^{k+1}\right\|
-\left\| p(y^{k+1})\right\|\\
&\geq &
\lambda \exp\left({\left|y^k_d\right|}-1\right)-
\left\| p(x^{k+1})\right\|-2\sqrt{d-1}\\
&\geq& \lambda \exp\left(\left|y^k_d\right|-1\right)-
\lambda \exp\left({\left|x^k_d\right|}-1\right)-2\sqrt{d-1}\\
&\geq& \lambda \exp\left(\left|y^k_d\right|-1\right)-
\lambda \exp\left({\left|y^k_d\right|}-M-1\right)-2\sqrt{d-1}\\
&=&
\frac{\lambda}{e} \left(1-e^{-M}\right)\exp\left|y^k_d\right|
-2\sqrt{d-1}.
\end{eqnarray*}
Noting that $|y^k_d|>M$ by (\ref{2c}) we find that if $M$
is sufficiently large then
$$\left|y^{k+1}_d\right|\geq 
\frac{\lambda }{3}\exp{\left|y^k_d\right|}+M.$$
Since
$$\frac{\lambda }{3}\exp{\left|y^k_d\right|}>
\frac{\lambda }{3}e^M\exp{\left|x^k_d\right|}
=\frac{\lambda e}{3}e^M\left\|x^{k+1}\right\|
\geq \frac{\lambda e}{3}e^M\left|x^{k+1}_d\right|
$$
the last inequality
in the conclusion of the lemma also holds
if $M$ is large.

\section{Proof of the Propositions}
\begin{proof}[Proof of Proposition 1]
For a sequence $\s=(s_k)$ in $\S$ we have
$$H(\s)=\bigcap_{ k\geq 0}
\left(\Lambda^{s_0}\circ\Lambda^{s_1}\circ
\ldots\circ\Lambda^{s_k}\right)(T(s_{k+1})).$$
Thus $X:= H(\s)\cup\{\infty\}$ is an intersection of nested,
connected, compact subsets of $\overline{\R^d}:=
\R^d\cup\{\infty\}$. This implies that $X$ is
compact and connected.

To prove that $H(\s)$ is a hair we follow 
Rottenfu{\ss}er, R\"uckert, Rempe and Schleicher~\cite{R3S} and
use the following lemma from~\cite{N}.
\begin{lemma}\label{lemma2}
Let $X$ be a non-empty, compact,
connected metric space. Suppose that there is a
strict linear ordering $\o$ on $X$ 
such that the order topology
on $X$ agrees with the metric topology. Then either $X$
consists of a single point or there is an order-preserving
homeomorphism from $X$ onto $[0,1]$.
\end{lemma}
To define the linear ordering on $X=H(\s)\cup\{\infty\}$
we choose $M$ according to Lemma~\ref{lemma1}.
For $x,y\in H(\s)$ we say that $x\o y$ if
there exists $k\geq 0$ such that $|y_d^k|>|x_d^k|+M$,
and we define $x\o\infty$ for all $x\in H(\s)$.
Lemma~\ref{lemma1} implies that $x\o y$ and $y\o x$ cannot
hold simultaneously.
Another easy consequence of Lemma~\ref{lemma1} is that
our relation $\o$ is transitive.

To show that it is a linear ordering we notice that
$\| x^k-y^k\|\geq\alpha^k\| x-y\|$ by~(\ref{2a}). 
Using~\eqref{2c1} we obtain
\begin{equation}\label{2c3}
|x^k_d-y^k_d|\geq
\left\|x^k-y^k\right\| -\left\|p(x^k)-p(y^k)\right\|
\geq \alpha^k\| x-y\|-2\sqrt{d-1}.
\end{equation}
Thus $x\neq y$ implies either $x\o y$ or $y\o x$.

Now we prove that the order topology on $X$
is the same as the topology induced from~$\overline{\R^d}$.
We have to show that the identity map 
from $X$ with the induced topology to
$X$
with the order topology is a homeomorphism.
Since $X$ with the induced topology is compact 
and since $X$ with the order topology
is Hausdorff, it suffices to show that the identity map
is continuous~\cite[p.~141, Theorem~8]{Kelley}.
Thus we only have to show that
the sets 
$$
U^-(a):=\left\{ w\in X:w\o a\right\}
\quad\text{and}\quad 
U^+(a):=\left\{ w\in X:w\oo a\right\}
$$
are open with respect to the induced topology for all $a\in X$.
In order to do so, let $w\in U^-(a)$ and choose the minimal $k$ such that
$|w^k_d|<|a^k_d|-M$.
Then there is a neighborhood $V$ of $w$ in $\R^d$ where
the same inequality is satisfied.
The intersection $V\cap H(\s)$ is a neighborhood
of $w$ that is contained in $U^-(a)$.
Thus  $U^-(a)$ is open with respect to the induced topology.
The proof for $U^+(a)$ is similar.

Thus the order topology on $X$
agrees with the topology induced from~$\overline{\R^d}$.
Proposition~\ref{prop1} now follows from Lemma~\ref{lemma2}.
\end{proof}

\begin{proof}[Proof of Proposition 2]
Let $y\in H(\s')\cap H(\s'')$.
Let $m$ be the smallest
subscript such that $s^\prime_m\neq s^{\prime\prime}_m$.
Then $f^m(y)$ belongs to the common boundary of
$T(s^\prime_m)$ and $T(s^{\prime\prime}_m)$.  From 
the definition of $f$
we conclude that $f^k(y)$ belongs to the hyperplane 
$\{x\in\R^d:x_d=0\}$
for all $k\geq m+1$.
This implies that $x\o y$ is impossible for any~$x$.
So $y$ is the  minimal element of the order $\o$
and thus an endpoint
of $H(\s')$ and $H(\s'')$.
\end{proof}

\begin{proof}[Proof of Proposition 3]
We follow the argument in~\cite{B} and with
$\psi:[1,\infty)\to\R$,
$$\psi(t):=\exp\left(\sqrt{\log t}\right)$$
and $M:=\max\{e,4\lambda\}$  we put
$$\Omega:=\left\{x\in\R^d:|x_d|\geq M,\ 
\|p(x)\| \leq
\psi\left(|x_d|\right)\right\}.$$
We then have
\begin{equation}\label{Me}
\|x\|\leq |x_d|+\|p(x)\|\leq|x_d|+\psi\left(|x_d|\right)\leq 2|x_d|,
\quad x\in\Omega.
\end{equation}
The following result is analogous to~\cite[Lemma 5.3]{B}.
\begin{lemma}\label{lemma3}
If $y\in H(\s)\backslash \{E(\s)\}$ then
$f^k(y)\to\infty$ as $k\to\infty$. Moroever, $f^k(y)\in\Omega$
for all large~$k$.
\end{lemma}
\begin{proof}
Let $\s=(s_k)_{k\geq 0}\in\S$ such
that $y\in H(\s)$. 
With $x=E(\s)$ and the ordering $\o$ as in Section~3 we have
$x\o y$. As before, we put 
$x^k=f^k(x)$ and $y^k=f^k(y)$ 
for $k\geq 0.$
By Lemma~\ref{lemma1} we have
$$|y^k_d|\geq 5|x^k_d|+M$$
for all large~$k$.
Using~\eqref{2c3} we see that
$| y_d^k|\to\infty$ 
and hence $y^k\to\infty$ as $k\to\infty$.

Since
$$\left\| p(y^k)\right\|\leq \left\| p(x^k)\right\|+2\sqrt{d-1}
\leq\left\|x^k\right\|+2\sqrt{d-1}$$
by~\eqref{2c1} we see that $f^k(y)\in \Omega$ holds for large~$k$ if 
$\| x^k\|\leq R$, where $R$ is any fixed constant.
Noting that 
$$\left\| x^k\right\|\leq \lambda \exp\left|x^{k-1}_d\right|\leq \lambda 
\exp\left\| x^{k-1}\right\|,$$ 
we also find that $f^k(y)\in \Omega$ holds for all large $k$ for which
$\| x^{k-1}\|\leq \log (R/\lambda) $.

We may thus suppose that $\min\{\| x^k\|,\|x^{k-1}\|\}$
is large.
Lemma~\ref{lemma1} now yields for large~$k$ that
\begin{eqnarray*}
\left|y^{k-1}_d\right|&\geq&\frac{\lambda }{3}\exp{\left|y^{k-2}_d\right|}+M\\
&\geq& \frac{\lambda }{3}\exp\left(5\left|x^{k-2}_d\right|+M\right)+M\\
&\geq&\frac{e^M}{3\lambda ^4}\left(\lambda \exp{\left|x^{k-2}_d\right|}\right)^5+M\\
&\geq&\left\|x^{k-1}\right\|^4\\
&\geq&\left|x^{k-1}_d\right|^4,
\end{eqnarray*}
and hence that
\begin{eqnarray*}
\left|y^k_d\right|&\geq&\frac{\lambda }{3}\exp{\left|y^{k-1}_d\right|}+M\\
&\geq&
\frac{\lambda }{3}\exp\left(\left|x^{k-1}_d\right|^4\right)\\
&\geq&\frac{\lambda }{3}\exp\left(\left(\log\left\| x^k\right\|\right)^4\right)\\
&\geq&
\exp\left(\left(\log\left\|x^k\right\|\right)^3\right).
\end{eqnarray*}
Thus
\begin{eqnarray*}
\left\| p(y^k)\right\|&\leq&\left\| p(x^k)\right\|+2\sqrt{d-1}\\
&\leq&\left\| x^k\right\|+ 2\sqrt{d-1}\\
&\leq&\exp\left(\left(\log\left| y_d^k\right|\right)^{1/3}\right)+2\sqrt{d-1}\\
&\leq& \exp\sqrt{\log\left|y_d^k\right|}
\end{eqnarray*}
for large~$k$.
This means that $y_k\in\Omega$, and the proof of 
Lemma~\ref{lemma3} is completed.
\end{proof}
The following result~\cite[Lemma~5.2]{B} is a simple consequence
of some classical covering lemmas. Here we denote by
$B(x,r)$ the open ball of radius $r$ around a point $x\in\R^d$.
\begin{lemma} \label{lemma5}
Let $Y\subset \R^d$ and $\rho>1$.
Suppose that for all $y\in Y$ and $\eta >0$ there exist
$r(y)\in (0,1)$, $d(y)\in (0,\eta)$ and $N(y)\in \N$ satisfying
$d(y)^\rho N(y)\leq r(y)^d$ such that $B(y,r(y))\cap Y$ can be
covered by $N(y)$ sets of diameter at most $d(y)$. Then
$\dim Y\leq\rho$.
\end{lemma}
In~\cite[Lemma~5.2]{B} it is additionally assumed that $Y$ 
is bounded, but this hypothesis can be omitted, since the
Hausdorff dimension of a set is the supremum of
the Hausdorff dimensions of its bounded subsets.

We now begin with the actual proof of Proposition~\ref{prop3},
following the argument  in~\cite{B}.
Since $f$ is locally bi-Lipschitz, and since the Hausdorff
dimension is invariant under bi-Lipschitz maps,
Lemma~\ref{lemma3} implies that it suffices to show that
$$Y:=\left\{y\in H(\s)\backslash \{E(\s)\}:
f^k(y)\in\Omega \text{ for all }k\geq 0\right\}$$
has Hausdorff dimension~$1$.
We shall prove this using Lemma~\ref{lemma5}.

Let $y\in Y\cap H(\s)$ and,
as before, put $y^k=f^k(y)$.
With $x=E(\s)$ we deduce from Lemma~\ref{lemma1} that
\begin{equation} \label{yk}
|y_d^{j+1}|>\frac{\lambda }{3}\exp{|y^j_d|}+M
\end{equation}
for large~$j$.

We now fix a large $k$ and 
denote by $B_k$ the closed ball of radius $\frac12 |y^k_d|$ around $y_k$. 
We cover $B_k\cap \Omega$ by closed
cubes of sidelength $1$
lying in $\{x\in \R^d:|x_d|\geq \frac12
|y^k_d|\}$.
If $c>2^{d-1}$, then the number $N_k$ of cubes required
satisfies 
$$N_k\leq c\; |y^k_d|\; \psi\left(2 \left|y^k_d\right|\right)^{d-1},$$
provided
$k$  is large enough.
Given $\varepsilon>0$ we thus can achieve that
\begin{equation} \label{Nk}
N_k\leq \left|y^k_d\right|^{1+\varepsilon} 
\end{equation}
by choosing $k$ large.

Let $B_0$ be the component
of $f^{-k}(B_k)$ that contains~$y$.
With 
$$\varphi:=\Lambda^{s_0}\circ \Lambda^{s_1}\circ 
\dots \circ \Lambda^{s_{k-1}}$$
we have $B_0=\varphi(B_k)$. 
Using~\eqref{n1} and~\eqref{n2} we find that if $C$ is one 
of the cubes of sidelength $1$ used to cover $B_k\cap \Omega$, 
then
$$
\diam \varphi(C)
\leq \frac{1}{\alpha^{k-1}}\frac{2}{\beta \left|y^k_d\right|}
\diam C \leq \frac{1}{\left|y^k_d\right|}
$$
if $k$ is sufficiently large.
Thus we can cover $B_0\cap Y$ by $N_k$ sets of diameter $d_k$, where
\begin{equation} \label{dk}
d_k\leq \frac{1}{\left|y^k_d\right|}.
\end{equation}
In order to apply Lemma~\ref{lemma5} we estimate the radius $r_k$
of the largest ball around $y$ that is contained in~$B_0$.
Let $z\in \partial B_0$ with $\|z-y\|=r_k$
and let $\sigma_0$ be the straight line connecting $y$ and~$z$.
For $1\leq j\leq k$ we put $\sigma_j=f^j(\sigma_0)$, $B_j=f^j(B_0)$
and $z^j=f^j(z)$.
Then $\sigma_k$ connects $y^k$ to $z^k\in \partial B_k$ and
thus 
\begin{equation} \label{lsk}
\length(\sigma_k)\geq \frac12\left|y_d^k\right|.
\end{equation}
We deduce from~\eqref{n2} that
$$
\diam B_{k-1}=
\diam  \Lambda^{s_{k-1}}\left(B_k)\right)
\leq \frac{2}{\beta \left|y^k_d\right|}
\diam  B_k =  \frac{2}{\beta}$$
and hence 
$$\diam  B_j\leq \frac{2}{\beta}$$
for $j\leq k-1$ by~\eqref{n1}.
Since $|y_d^j|\geq M>4/\beta$ this implies that
$$\sigma_j\subset
B_j\subset B\left(y^j,\frac12\left|y_d^j\right|\right)
\subset B\left(y^j,\frac12\left\|y^j\right\|\right)$$
for $j\leq k-1$.
It thus follows from~\eqref{n3} and~\eqref{Me} that
$$\length \sigma_{j}
=\length  \Lambda^{s_{j}}(\sigma_{j+1})
\geq  \frac{2\delta}{3\left\|y^{j+1}\right\|}
\length \sigma_{j+1}
\geq  \frac{\delta}{3\left|y_d^{j+1}\right|}
\length \sigma_{j+1}$$
for $j\leq k-1$ and this implies that
$$
\length \sigma_k
\leq 
\left(\frac{3}{\delta}\right)^k 
\left(\prod_{j=1}^{k} \left|y_d^j\right|\right)
 \length \sigma_0
$$
Combining this with~\eqref{lsk}
we find that
$$r_k=\length \sigma_0
\geq \frac12 
\left(\frac{\delta}{3}\right)^k\frac{1}{\prod_{j=1}^{k-1} \left|y_d^j\right|}.
$$
Using~\eqref{yk} we see that we can achieve
\begin{equation} \label{rk}
r_k\geq \frac{1}{\left|y_d^k\right|^{\varepsilon}}
\end{equation}
by choosing $k$ large.

We thus find that we can cover $B(y,r_k)\cap Y$ 
by $N_k$ sets of diameter at most $d_k$, where 
$N_k$, $d_k$ and $r_k$ satisfy \eqref{Nk}, \eqref{dk}
and~\eqref{rk}.
With $\rho=1+(d+1)\varepsilon$ it follows from \eqref{Nk}, \eqref{dk}
and~\eqref{rk}
that
$$(d_k)^\rho N_k \leq \left|y_d^k\right|^{1+\varepsilon-\rho}
= \left|y_d^k\right|^{-d \varepsilon}\leq (r_k)^d.$$
Given $\eta>0$ we can also achieve that $r_k<1$ and $d_k<\eta$ by 
choosing $k$ large.
We thus see that the hypothesis of Lemma~\ref{lemma5}
are satisfied with $r(y)=r_k$, $d(y)=d_k$ and $N(y)=N_k$.

It follows that $\dim Y\leq \rho=1+(d+1)\varepsilon$. Since
$\varepsilon>0$ was arbitrary, we obtain  $\dim Y\leq  1$.
\end{proof}
\section{An explicit bi-Lipschitz map}
\label{explicit}
Let $B^+:=[-1,1]^{d-1}\times [0,1]$,
$B^-:=[-1,1]^{d-1}\times [-1,0]$,
$U^+:=\{x\in\R^d:\|x\|_2 \leq 1,x_d\geq 0\}$
and
$U^-:=\{x\in\R^d:\|x\|_2 \leq 1,x_d\leq 0\}$.
Then $h_1:=B^+\to B^-$, $x \mapsto x-
(0,\dots,0,1)$,
and $h_2:B^-\to U^-$, $x\mapsto (\|x\|_\infty/\|x\|_2)x$, are
both bi-Lipschitz, and with $X:=[-1,1]^{d-1}\times \{1\}$ and
$Y:=\{x\in\R^d:\|x\|_2 \leq 1,x_d=0\}$ we have
$h_2(h_1(X))=Y$.
It remains to  define a bi-Lipschitz map $h_3:U^-\to U^+$
with $h_3(Y)=\{x\in\R^d:\|x\|_2 =1,x_d\geq 0\}$.
Then $h:=h_3\circ h_2\circ h_1$ has the desired properties.

In order to define $h_3$ we note that 
$$T(z)=\frac{z+i}{iz+1}$$
defines 
a bi-Lipschitz
map from the lower half-disk
$\{z\in \C:|z|\leq 1,\im z\leq 0\}$
to the upper half-disc
 $\{z\in \C:|z|\leq 1,\im z\geq 0\}$, 
with $\{z\in \C:|z|\leq 1,\im z=0\}$
being mapped onto $\{z\in \C:|z|=1,\im z\geq 0\}$.
With $x=(x_1,\dots,x_d)=(p(x),x_d)$  and
$z=\|p(x)\|_2+ix_d$ it follows that 
$$h_3(x)=\left(\frac{p(x)}{\|p(x)\|_2} \re T(z), \im T(z)\right)$$
has the desired properties.
\section{Quasiregular maps}\label{quasi}
Let $\Omega\subset \R^d$ be open. A continuous map 
$f:\Omega\to\R^d$ is called quasi\-regular 
if it belongs to the Sobolev space
$W^1_{d,\loc}(\Omega)$ and if   there exists a
constant $K_O\geq 1$ such that
\begin{equation}
\label{dil}
\| DF(x)\|^d\leq K_O\,J_F(x)\quad\mbox{a.e.},
\end{equation}
where 
$J_F=\det DF$ denotes the Jacobian determinant.
Equivalently, there exists $K_I\geq 0$ such that
\begin{equation}
\label{dil2}
J_F(x)\leq K_I\,\ell(DF(x))^d\quad\mbox{a.e.}
\end{equation}
The smallest constants $K_O$ and $K_I$ for which the above estimates
hold are called the outer and inner dilatation. 
For  a thorough treatment of quasi\-regular maps we
refer to~\cite{Rick}.

To see that our map $F$ defined in Section~\ref{intro}
is quasi\-regular, we note that 
first that~\eqref{dil} holds on the half-cube
$(-1,1)^{d-1}\times (0,1)$
since $F$ is bi-Lipschitz there. 
By the same reason, ~\eqref{dil} holds on the 
bounded set  $(-1,1)^{d-1}\times (1,2)$.
Using~\eqref{DFexp} we deduce that~\eqref{dil} 
holds on $(-1,1)^{d-1}\times (1,\infty)$.
Thus~\eqref{dil} holds on $(-1,1)^{d-1}\times (0,\infty)$ and
in the sets obtained from this by reflection.
We deduce that $F$ is indeed quasi\-regular.

We mention that it follows from~\eqref{dil} and~\eqref{dil2}
that if $F$ is quasiregular, then 
\begin{equation}
\label{dil3}
\| DF(x)\|\leq K\: \ell(DF(x))\quad\mbox{a.e.}
\end{equation}
where $K=(K_OK_I)^{1/d}$.
We could also use~\eqref{dil3} 
instead of~\eqref{dil} or~\eqref{dil2}
in the definition of quasiregularity.
It follows from~\eqref{n2} and~\eqref{n3} that~\eqref{dil3}
holds for $|x_d|\geq 1$ with $K=1/(\beta\delta)$.
This one reason why we said in the introduction
that the quasiregularity of~$F$
is among the underlying ideas of the proof.

We note that for quasi\-regular maps there is no obvious definition
of the Julia set; see, however,~\cite{Be10,SunYang00}.
On the other hand, the escaping set $I(f)$ can be 
defined. It was shown 
in~\cite{Bergweiler08} that if $f$ is a
quasi\-regular self-map of~$\R^d$ with
an essential singularity at~$\infty$,
then $I(f)\neq\emptyset$. In fact, $I(f)$ has an unbounded
component. Fletcher and Nicks~\cite{FletcherNicks} have shown
that for quasiregular maps of polynomial type the 
boundary of the escaping set has properties similar to the
Julia set of polynomials.

We mention that for the entire functions $f(z)=\lambda e^z$ or
$\lambda\sin z+\mu$ considered in Theorems A--G
we have $I(f)\subset J(f)$ and
thus $J(f)=\overline{I(f)}$; see~\cite[Theorem~1]{Eremenko92}.
This plays an important role in the proofs of these theorems.
For example, McMullen actually proved that the conclusion of
Theorems~B and E holds with $J(f)$ replaced by $I(f)$.
Also, a crucial part in the proofs of Theorems~C, F and G is
based on the fact that points which are on a hair but which 
are not endpoints escape to
infinity under iteration very fast.

This also played an important role in our proof.
In particular,
for the map $f$ considered in this paper we have
$$\bigcup_{\s\in \S} H(\s)\backslash \{ E(\s)\}\subset I(f)$$
by Lemma~\ref{lemma3}.
On the other hand, it is not difficult to see that
$\{E(\s): \s\in \S\}$ intersects both $I(f)$ and the 
complement of $I(f)$.

\end{document}